\documentclass{amsart}

\newtheorem*{theorem}{Theorem}
\newtheorem*{corollary}{Corollary}
\newtheorem{lemma}{Lemma}

\newtheorem*{proposition}{Proposition}

\theoremstyle{remark}
\newtheorem{remark}{Remark}

\usepackage{graphicx}

\newcommand{\angleb}[1]{\langle #1 \rangle} 
\newcommand{\circleb}[1]{\langle #1 \rangle^{\circ}} 

\begin{document}

\title{Pairs of boundary slopes with small differences}

\author{Kazuhiro Ichihara}

\address{Department of Mathematics, 
College of Humanities and Sciences, Nihon University,
3-25-40 Sakurajosui, Setagaya-ku, Tokyo 156-8550, Japan}

\email{ichihara@math.chs.nihon-u.ac.jp}

\thanks{The author is partially supported by
Grant-in-Aid for Young Scientists (B), No.~23740061, 
Ministry of Education, Culture, Sports, Science and Technology, Japan.}

\dedicatory{Dedicated to Professor Fico Gonz\'{a}lez Acu\~{n}a on the occasion of his 70th birthday}

\begin{abstract}
We show that, for any positive real number, 
there exists a knot in the 3-sphere 
admitting a pair of boundary slopes 
whose difference is at most the given number. 
\end{abstract}



\date{\today}

\maketitle

\section{Introduction}

In this paper we give somewhat curious examples 
of paris of boundary slopes for knots in the 3-sphere $S^3$. 
The \textit{boundary slope} of an essential surface $F$ in a knot exterior 
is defined as the slope 
represented by the boundary of $F$ 
on the peripheral torus of the knot.  
Note that such slopes for a knot in $S^3$ 
are naturally parametrized by rational numbers.

In \cite{CS84},  for each non-trivial knot, Culler and Shalen proved that 
there always exist at least two boundary slopes. 
Thus we can take a pair of boundary slopes, 
and consider their difference as a pair of rational numbers. 
Actually, in \cite{CS99}, they showed that, 
for any non-trivial knot in $S^3$ not having the meridional boundary slope, 
there always exists a pair of boundary slopes 
whose difference is at least 2. 

It is known that this lower bound can be improved for some class of knots. 
For example, 
as observed in \cite[Section 1]{IMcd},  
any alternating knot admits a pair of boundary slopes 
whose difference is bounded from below in terms of its crossing number. 
As another example, in \cite{CS04}, 
it was shown that, 
if a knot exterior contains only two essential surfaces, 
then the difference between their boundary slopes 
are bounded from below in terms of the Euler characteristics and the number of sheets 
(with some exceptional cases). 
Furthermore, in \cite{IMlb}, 
the author together with Mizushima showed that, 
for any nontrivial Montesinos knot, 
there exists a pair of boundary slopes 
whose difference is bounded from below similarly. 

In contrast to these results, in this paper, we show that 
there exists pairs of boundary slopes for knots with arbitrarily small differences. 

\begin{corollary}
For any positive number $\varepsilon$, 
there exists a knot in $S^3$ 
admitting a pair of boundary slopes 
whose difference is at most $\varepsilon$. 
\end{corollary}

Actually we give a concrete examples of a sequence of knots 
with paris of boundary slopes becoming arbitrarily close to each other. 

\begin{theorem}
The Montesinos knot $K_n = M( -1/2 , 2/5 , 1/n)$ with an odd positive integer $n \ge 11$  
admits 
a pair of boundary slopes $2(n-1)^2/n$ and $2(n^2 - 9n +15)/(n-7)$. 
\end{theorem}

\begin{proof}[Proof of Corollary]
Consider the sequence of the Montesinos knots $ \{ K_n = M( -1/2 , 2/5 , 1/n) \}$ 
for odd positive integers $n \ge 11$. 
Let $r_n = 2(n-1)^2/n$ and $r'_n = 2(n^2 - 9n +15)/(n-7)$ 
be the pairs of boundary slopes for $K_n$, 
whose existence is guaranteed by the theorem above.
Then their difference is calculated as 
$ r'_n - r_n = 2 (n^2 - 9n +15)/(n-7) - 2(n-1)^2/n  = 2 ( 1/(n-7) - 1/n )$. 
Thus, as $ n \to \infty$, the differences $ | r'_n - r_n |$ converge to $0$. 
\end{proof}

To prove the theorem, we depend on the machinery 
to enumerate the boundary slopes for Montesinos knots, 
which was developed by Hatcher and Oertel in \cite{HO}. 
In the next section, we will briefly recall the machinery. 
Based on this, in Sect.~\ref{sec:3}, we will give a proof of our theorem, and 
in the last section, we will give supplementary data for the surfaces 
with the desired boundary slopes. 

We here note that to find the example, 
the key ingredient was the computer program \cite{D} 
written by Dunfield, which implements the algorithm given in \cite{HO}.

\section{Boundary slopes for Montesinos knots}\label{sec:2}

In this section, after setting up our notation and definition, 
we briefly recall the machinery developed by Hatcher and Oertel in \cite{HO} 
to enumerate the boundary slopes for Montesinos knots. 
See \cite{HO} for a fundamental reference, and 
also see \cite{IMb}, \cite{IMcd}, \cite{IMlb}, \cite{IMcc}, \cite{IMpn} and \cite{{Callahan13}} 
for related results and detailed explanations.

\subsection{Essential surface and Montesinos knot}
We start with recalling basic definitions and notations. 
For example, see the book \cite{R} as a detailed reference. 

Let $K$ be a knot (i.e., an embedded circle) in the 3-sphere $S^3$. 
Denote by $E(K)$ the exterior of $K$ in $S^3$, 
meaning that, $E(K)$ is the complement of 
an open tubular neighborhood of $K$ in $S^3$. 
A compact surface $F$ with non-empty boundary $\partial F$ 
properly embedded in $E(K)$ 
is called \textit{essential} if it is incompressible and boundary-incompressible. 
The \textit{boundary slope} of $F$ is defined as the slope 
(i.e., an isotopy class of a non-trivial unoriented simple closed curve) 
represented by $\partial F$ on the torus $\partial E(K)$. 
Such slopes for a knot in $S^3$ 
are naturally parametrized by rational numbers 
by using a standard meridian-preferred longitude system. 
In particular, when an irreducible fraction $p/q$ 
corresponds to a slope $\gamma$, 
then $\gamma$ minimally intersects 
with the meridian in $|q|$ times and the longitude $|p|$ times. 

A knot in $S^3$ is called a \textit{Montesinos knot} 
if the knots is composed of a number of rational tangles. 
We here omit various properties of Montesinos knots.

\subsection{Hatcher-Oertel's machinery}

We here briefly recall the machinery developed by \cite{HO}, 
which actually gives an algorithm 
to enumerate all the boundary slopes for a given Montesinos knot. 
See \cite{IMb}, \cite{IMcd}, \cite{IMlb}, \cite{IMcc}, \cite{IMpn} and \cite{{Callahan13}} also.

\subsubsection{Edgepath system}

The heart of the Hatcher-Oertel's machinery would be 
using combinatorial objects, called ``edgepath systems'', 
to describe properly embedded surfaces in Montesinos knots exteriors. 
In particular, 
for an edgepath system satisfying certain conditions, 
one can construct an essential surface 
properly embedded in the given Montesinos knot exterior. 
Furthermore, from some combinatorial data of an edgepath system, 
we can compute the boundary slope, the Euler characteristic, and 
the number of boundary components of the surface so obtained.

An \textit{edgepath system} is defined as a finite collection of edgepaths, and 
an \textit{edgepath} is defined as a finite path 
lying on a special kind of an embedded graph on a plane, 
which we call the \textit{diagram}, denoted by $\mathcal{D}$. 
Note that we allow a single point on an edge as an edgepath, and 
regard each edgepath running from right to left.

  \begin{figure}[htb]
   \begin{center}
    \begin{picture}(150,210)
     \put(0,0){\includegraphics{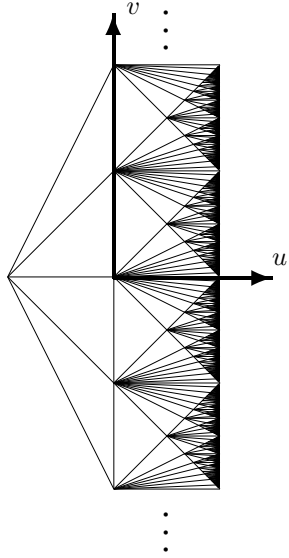}}
     \put(77,185){\rotatebox{90}{\scalebox{1.5}{$\cdots$}}}
     \put(77,-5){\rotatebox{90}{\scalebox{1.5}{$\cdots$}}}
     \put(60,100){\scalebox{2.0}{\vector(1,0){30}}}
     \put(60,100){\scalebox{2.0}{\vector(0,1){50}}}
     \put(120,105){$u$}
     \put(65,200){$v$}
    \end{picture}
   \end{center}
   \caption{The diagram $\mathcal{D}$}
   \label{Fig:Diagram}
  \end{figure}

The diagram $\mathcal{D}$, which is an embedded graph in the $uv$-plane, 
is precisely described as follows. 
For an irreducible fraction $p/q$, 
a vertex of $\mathcal{D}$ is set as: 
either a point $\angleb{p/q}=((q-1)/q,p/q)$,
a point $\circleb{p/q}=(1,p/q)$, or a point $\angleb{\infty}=(-1,0)$. 
If a pair of vertices $\angleb{p/q}$ and $\angleb{r/s}$ satisfy the condition $|ps-qr|=1$, 
(regarding $\infty$ as $1/0$), 
then the two vertices are connected by an edge. 
Such an edge is denoted by $\angleb{p/q} \ - \ \angleb{r/s}$.

\subsubsection{Rational point}

An edge path consisting of a single point is called a \textit{constant edgepath}. 
Also we can consider edgepaths ended at some interior points on edges. 
Such a point in the interior of an edge 
is expressed as a rational point on an edge as follows. 

Let $e$ be an edge $\angleb{p/q} \ - \ \angleb{r/s}$ with $q\ge 1$ and $s\ge 1$, 
and $k$ and $l$ positive integers. 
Then, $\frac{k}{k+l}\angleb{p/q}+\frac{l}{k+l}\angleb{r/s}$ 
denotes a point on $e$ with $uv$-coordinates 
\begin{equation}\label{uvcood1}
\frac{kq}{kq+ls} \left( \frac{q-1}{q} , \frac{p}{q} \right) + \frac{ls}{kq+ls} \left( \frac{s-1}{s} , \frac{r}{s} \right) =
\left(  \frac{k q+ls -( k+l )}{kq+ls} , \frac{kp + lr}{kq+ls} \right)
\end{equation}

Also, $\frac{k}{k+l}\angleb{p/q}+\frac{l}{k+l}\circleb{p/q}$ 
denotes a point on $e$ with $uv$-coordinates 
\begin{equation}\label{uvcood2}
\frac{k}{k+l} \left( \frac{q-1}{q} , \frac{p}{q} \right) + \frac{l}{k+l} \left( 1 , \frac{p}{q} \right) =
\left(  \frac{kq +lq -k}{kq+lq} , \frac{p}{q} \right)
\end{equation}

These might seem to be somewhat unnatural formula, 
but, they actually reflects certain natural information of the surface so constructed. 
See \cite{HO} or \cite{IMb} for example.

\subsubsection{From edgepath system to surface}

For an edgepath system, under certain conditions, 
one can obtain a properly embedded surface in a Montesinos knot exterior. 
The following lemma gives such conditions. 
See \cite[Section 1, Proposition 1.1]{HO} and \cite[Section 2]{IMb} for a proof for example.

\begin{lemma}\label{CandidateSurface}
For a Montesinos knot $K=M(R_1, \cdots , R_N)$, 
a properly embedded surface in the exterior of $K$ is constructed 
corresponding to an edgepath system satisfying the following conditions: 

For an edgepath system $\Gamma=(\gamma_1, \cdots ,\gamma_N)$, 
\begin{itemize}
\item[(E1)] the starting point of $\gamma_i$ lies on the edge $\angleb{R_i}$\,--\,$\circleb{R_i}$, and if this starting point is not the vertex $\angleb{R_i}$, then the edgepath $\gamma_i$ is constant 
on the horizontal edge $\angleb{R_i} \ - \ \circleb{R_i}$,
\item[(E2)] each $\gamma_i$ is minimal, i.e., it never stops and retraces itself, nor does it ever go along two sides of the same triangle of $\mathcal{D}$ in succession,
\item[(E3)] all the ending points of the $\gamma_i$'s are lying on one vertical line in $\mathcal{D}$, and whose vertical coordinates add up to zero, and 
\item[(E4)] each $\gamma_i$ proceeds monotonically from right to left; ``monotonically'' means in a weak sense that motion along vertical edges is permitted. 
\end{itemize}
Conversely, 
every essential surface in a Montesinos knot exterior is expressed by an edgepath system in such a way. 
\end{lemma}

A surface represented by an edgepath system as in Lemma \ref{CandidateSurface} 
is called a \textit{candidate surface}.

\subsubsection{Sign, length and twist}

The boundary slope of a candidate surface 
corresponding to an edgepath system is calculated as follows.

We first define the \textit{sign} of an edge $e$ in $\mathcal{D}$ 
not connected to $\angleb{ \infty }$, which we denote by $\sigma(e)$. 
Recall that we are assuming that 
an edge in $\mathcal{D}$ is oriented from right to left. 
Then the edge is said to be \textit{increasing} (respectively \textit{decreasing}) 
if the $v$-coordinate increases (resp. decreases) 
as a point moves along the edge in that direction.
According to whether $e$ is increasing or decreasing, 
we assign $+1$ or $-1$ to an edge $e$ as the sign $\sigma(e)$ respectively. 

We next define the \textit{length} of an edge $e$ in $\mathcal{D}$, denoted by $|e|$. 
The length of a complete edge is set to be $1$. 
The length of an edge such as 
$(\frac{k}{k+l}\angleb{p/q}+\frac{l}{k+l}\angleb{r/s}) \ - \ \angleb{r/s}$ 
is set to be $\frac{k}{k+l}$. 

Now, for a non-constant edge $e$ on $\mathcal{D}$ in the region with $u>0$, 
we define the \textit{twist} $\tau(e)$ of $e$ as $-2\,\sigma(e)~|e| $. 
For the other kind of edges, its twist is set to be $0$. 
Taking sum of all the edges included, 
we define the twist $\tau(\Gamma)$ of an edgepath system $\Gamma$, and 
equivalently, the twist $\tau(F)$ of a candidate surface $F$ corresponding to an edgepath system. 
In general, plural candidate surfaces correspond to an edgepath system, 
but the twist is well-defined.  See \cite{HO} or \cite{IMb}. 

Under these settings, as shown in \cite{HO}, or as explained in \cite{IMb}, 
we can calculate the boundary slope of a candidate surface as follows. 

\begin{lemma}\label{lemBS}
Let $F$ be a candidate surface for a Montesinos knot. 
The boundary slope $r$ of $F$ is calculated as $r = \tau(F) - \tau(F_S)$, 
where $F_S$ denotes a Seifert surface in the list of candidate surfaces of the knot. 
\end{lemma}

\subsubsection{Incompressibility}

In the next section, to prove our theorem, 
we will give a pair of edgepath systems 
satisfying all the conditions in Lemma \ref{CandidateSurface}. 
Then we need to check that the surfaces so obtained are actually essential. 
To do this, we will use two lemmas, essentially obtained in \cite{HO}.

To state the lemmas, we need some more definitions. 
Assume that the $u$-coordinate of the left endpoint of an edgepath is $u_0$.
Then the edgepath is said to be of \textit{type I}, \textit{type II} or \textit{type III} 
if $u_0$ satisfies $u_0>0$, $u_0=0$ or $u_0<0$, respectively. 


The next lemma is a version of \cite[Lemma 2.1(1)]{IMlb}. 

\begin{lemma}\label{lemTypeI}
For a type I edgepath system, 
if all the last edges of the edgepaths in the edgepath system have a common sign, 
then all the candidate surfaces for the edgepath system are essential. 
\end{lemma}

\begin{proof}
By Corollary 2.4 and Propositions 2.6, 2.7 and 2.8(a) in \cite{HO}, 
if a type I edgepath system 
represents an inessential surface, 
then its cycle of final $r$-values must include both positive and negative terms. 
We here omit the precise definition of $r$-values, 
but this is equivalent to that 
there exist a pair of last edges of edgepaths in the edgepath system having different signs. 
See \cite[Subsection 2.3]{IMlb} for details. 
\end{proof}

The next lemma is just the Proposition 2.1 in \cite{HO}.

\begin{lemma}\label{lemconst}
For a type I edgepath system, 
if at least one of its edgepaths is constant, 
then all the candidate surfaces for the edgepath system are essential. 
\end{lemma}

\section{Proof of Theorem}\label{sec:3}

We first prepare an edgepath system representing a Seifert surface, and 
calculate its twist. 

\begin{lemma}\label{lemSs}
The edgepath system $\Gamma_S$ given by 
$$
\begin{array}{l}
\delta_1 : 
\angleb{\infty} \ - \ \angleb{-1} \ - \ \angleb{-\dfrac{1}{2}} \\[10pt]
\delta_2 : 
\angleb{\infty} \ - \ \angleb{0} \ - \ \angleb{\dfrac{1}{2}} \ - \ \angleb{\dfrac{2}{5}} \\[10pt]
\delta_3 : 
\angleb{\infty} \ - \ \angleb{1} \ - \ \angleb{\dfrac{1}{2}} \ - \ 
\cdots \ - \ \angleb{\dfrac{1}{n-1}} \ - \ \angleb{\dfrac{1}{n}}
\end{array}
$$
represents a Seifert surface of 
the Montesinos knot $K_n = M( -1/2 , 2/5 , 1/n)$, and its twist is equal to $4-2n$. 
\end{lemma}

\begin{proof}
We here check that $\Gamma_S$ satisfies two conditions 
to represent a Seifert surface described in 
\cite[{Section 1, subsection ``Computing $\partial$-Slopes'' }]{HO}. 

For the first condition, 
we take the mod 2 values of $p$ and $q$ in all vertices $\angleb{\frac{p}{q}}$, 
and reduce the edges in $\Gamma_S$ to one of the three types, 
say $\angleb{\frac{odd}{even}} \ - \ \angleb{\frac{even}{odd}}$, 
$\angleb{\frac{even}{odd}} \ - \ \angleb{\frac{odd}{odd}}$, and 
$\angleb{\frac{odd}{even}} \ - \ \angleb{\frac{odd}{odd}}$. 
(Note that we regard $\infty$ as $1/0$.) 
The condition we have to check is that 
each of edgepaths in $\Gamma_S$ uses only one of the three types. 
Then we see that 
$\delta_1$ (respectively, $\delta_2$, $\delta_3$) 
uses only the edges of type $\angleb{\frac{odd}{even}} \ - \ \angleb{\frac{odd}{odd}}$ 
(resp. $\angleb{\frac{odd}{even}} \ - \ \angleb{\frac{even}{odd}}$, 
$\angleb{\frac{odd}{even}} \ - \ \angleb{\frac{odd}{odd}}$). 
It follows that $\Gamma_S$ satisfies the first condition. 

The second condition which we have to check is that 
$\Gamma_S$ contains even number of edgepaths with penultimate vertices corresponding to odd-integers. 
Here, for type III edgepaths, penultimate vertices mean the vertices on the vertical $v$-axis. 
Actually we see that 
$\delta_1$ and $\delta_3$ have the penultimate vertices $\angleb{-1}$ and $\angleb{1}$, 
which correspond to odd integers, 
while $\delta_2$ has the penultimate vertex $\angleb{0}$, which does an even integers. 
It follows that $\Gamma_S$ satisfies the second condition. 

To calculate the twists, it suffices to count 
the number of increasing/decreasing edges in $\Gamma_S$. 
Actually we see that the number of increasing (respectively decreasing) edges in $\Gamma_S$ 
is just $n$ (resp. 2). 
Thus the twist $\tau(\Gamma_S) = -2 (n-2) = 4-2n$. 
\end{proof}

Now let us give a proof of the main theorem. 

\begin{proof}[Proof of Theorem]
Let $K_n$ be the Montesinos knot $M( -1/2 , 2/5 , 1/n)$ with an odd positive integer $n \ge 11$. 

We give a pair of edgepath systems for $K_n$, and 
show that they represent essential surfaces properly embedded in $E(K_n)$ with 
boundary slopes $2(n-1)^2/n$ and $2(n^2 - 9n +15)/(n-7)$. 

The first edgepath system $\Gamma = ( \gamma_1 , \gamma_2 , \gamma_3 )$ is 
given as follows. 
$$
\begin{array}{l}
\gamma_1 : 
\left( \dfrac{1}{n} \angleb{-1} + \dfrac{n-1}{n} \angleb{-\dfrac{1}{2}} \right) 
\ - \ \angleb{-\dfrac{1}{2}} \\[10pt]
\gamma_2 : 
\left( \dfrac{1}{n} \angleb{0} + \dfrac{n-1}{n} \angleb{\dfrac{1}{2}} \right) 
\ - \ \angleb{\dfrac{1}{2}} \ - \ \angleb{\dfrac{2}{5}} \\[10pt]
\gamma_3 : 
\left( \dfrac{n-1}{n} \angleb{0} + \dfrac{1}{n} \angleb{\dfrac{1}{n}} \right) 
\ - \ \angleb{\dfrac{1}{n}} 
\end{array}
$$
This clearly shows that $\Gamma$ satisfies (E1), (E2) and (E4) in Lemma \ref{CandidateSurface}. 

The $uv$-coordinates of their endpoints are calculated by \eqref{uvcood1} as follows. 
$$
\gamma_1 : 
\left( \dfrac{n-1}{2n-1} , \dfrac{-n}{2n-1} \right) , \qquad 
\gamma_2 : 
\left( \dfrac{n-1}{2n-1} , \dfrac{n-1}{2n-1} \right) , \qquad 
\gamma_3 : 
\left( \dfrac{n-1}{2n-1} , \dfrac{1}{2n-1} \right) 
$$

From these, we see that the condition (E3) in Lemma \ref{CandidateSurface} 
is satisfied for $\Gamma$. 
Thus we have a candidate surface, say $F$, for $\Gamma$ in the exterior $E(K_n)$. 

Since $\Gamma$ is of type I and all the last edges of edgepaths have a common sign, 
by Lemma \ref{lemTypeI}, the surface $F$ is essential. 

Now let us calculate the boundary slope of $F$. 
By setting 
$$
\begin{array}{l}
e_1 : 
\left( \dfrac{1}{n} \angleb{-1} + \dfrac{n-1}{n} \angleb{-\dfrac{1}{2}} \right) 
\ - \ \angleb{-\dfrac{1}{2}} \\[10pt]
e_{2,1} : 
\angleb{\dfrac{1}{2}} \ - \ \angleb{-\dfrac{2}{5}} \\[10pt]
e_{2,2} : 
\left( \dfrac{1}{n} \angleb{0} + \dfrac{n-1}{n} \angleb{-\dfrac{1}{2}} \right) 
\ - \ \angleb{\dfrac{1}{2}}\\[10pt]
e_3 : 
\left( \dfrac{n-1}{n} \angleb{0} + \dfrac{1}{n} \angleb{\dfrac{1}{n}} \right) 
\ - \ \angleb{\dfrac{1}{n}} 
\end{array}
$$
the twists of $F$ is obtained as follows. 
$$
 -2\, ( \sigma(e_1)~|e_1| + \sigma(e_{2,1})~|e_{2,1}| + \sigma(e_{2,2})~|e_{2,2}| + \sigma(e_3)~|e_3| )
=
 -2\, \left( -\frac{1}{n} + 1 - \frac{1}{n} - \frac{n-1}{n} \right)
 =\frac{2}{n}
$$

Thus the boundary slope $r$ of $F$ is calculated as 
$$
r = \frac{2}{n} - 2(2-n) = \frac{2 (n-1)^2}{n}
$$
by Lemmas \ref{lemBS} and \ref{lemSs}. 

\bigskip

The second edgepath system 
$\Gamma' = ( \gamma'_1 , \gamma'_2 , \gamma'_3 )$ is given as the following. 
$$
\begin{array}{l}
\gamma'_1 : 
\dfrac{n-7}{n-4} \angleb{-\dfrac{1}{2}} + \dfrac{3}{n-4} \circleb{-\dfrac{1}{2}} \\[10pt]
\gamma'_2 : 
\left( \dfrac{n-9}{n-7} \angleb{\dfrac{1}{2}} + \dfrac{2}{n-7} \angleb{\dfrac{2}{5}} \right) 
\ - \ \angleb{\dfrac{2}{5}} \\[10pt]
\gamma'_3 : 
\left( \dfrac{n-8}{n-7} \angleb{0} + \dfrac{1}{n-7} \angleb{\dfrac{1}{n}} \right) 
\ - \ \angleb{\dfrac{1}{n}} 
\end{array}
$$
This clearly shows that $\Gamma'$ satisfies (E1), (E2) and (E4) in Lemma \ref{CandidateSurface}. 

The $uv$-coordinates of endpoints are calculated by \eqref{uvcood1} and \eqref{uvcood2} as follows. 
$$
\gamma'_1 : 
\left( \dfrac{n-1}{2n-8} , - \dfrac{1}{2} \right) , \qquad 
\gamma'_2 : 
\left( \dfrac{n-1}{2n-8} , \dfrac{n-5}{2n-8} \right) , \qquad 
\gamma'_3 : 
\left( \dfrac{n-1}{2n-8} , \dfrac{1}{2n-8} \right) 
$$

From these, we see that the condition (E3) in Lemma \ref{CandidateSurface} 
is satisfied for $\Gamma'$. 
Thus we have a candidate surface, say $F'$, for $\Gamma'$ in the exterior $E(K_n)$. 

Since $\Gamma'$ contains a constant edgepath, 
the surface $F'$ is essential by Lemma \ref{lemconst}. 

Now let us calculate the boundary slope of $F'$. 
Since $\gamma'_1$ is a constant edge path, it has no contribution to twists. 
The edgepaths $\gamma'_2$ and $\gamma'_3$ both consist of only single edges, 
the twists of $F'$ is obtained as follows. 
$$
\tau(F') =
 -2\, ( \sigma(\gamma'_2)~|\gamma'_2| + \sigma(\gamma'_3)~|\gamma'_3| )
=
 -2\, \left( \frac{n-9}{n-7}  - \frac{n-8}{n-7}  \right)
 =\frac{2}{n-7}
$$

Thus the boundary slope $r'$ of $F'$ is calculated as 
$$
r' = \frac{2}{n-7} - 2(2-n) = \frac{2 (n^2 - 9n +15)}{n-7} 
$$
by Lemmas \ref{lemBS} and \ref{lemSs}. 

This completes the proof of the theorem. 
\end{proof}

\begin{remark}\label{rmk1}
The essential surfaces obtained in the proof above 
were essentially treated in \cite[4.1.3. (7) and (8)]{IMb}. 
\end{remark}

\section{Denominators and Euler characteristics}\label{sec:4}

In this section, as an addendum, we include some calculations about the pair of surfaces we found. 

\subsection{The numbers of sheets and the numbers of boundary components}

If an essential surface meets a small meridional circle of a Montesinos knot minimally in $m$ points, 
then the {\em number of sheets} of the surface is defined as $m$. 
When we construct a surface from an edgepath system,
the number of sheets of the surface denoted by $\sharp s$ is determined as follows.

When the last edge of an edgepath in an edgepath system 
is a partial edge of length $k/(k+l)$ with the fraction $k/(k+l)$ irreducible, 
then $\sharp s$ has to be a multiple of $k+l$.
When an edgepath in an edgepath system is a constant edgepath 
such as $\frac{k}{k+l} \angleb{p/q} \ - \ \frac{l}{k+l} \circleb{p/q}$, 
then $\sharp s$ has to be a multiple of $k$. 
Consequently $\sharp s$ is determined as the least common multiple of these integers. 

In our settings, from this, we immediately obtain the following.

\begin{proposition}
Let $F$ and $F'$ be essential surfaces constructed in the proof of Theorem. 
Then their numbers of sheets are $n$ and $n-7$, respectively. \qed
\end{proposition}

For an essential surface in a knot exterior, 
the number of sheets is equal to the product of 
the number of boundary components and the denominator of the boundary slope. 
Thus, in our setting, we have: 

\begin{corollary}
Let $F$ and $F'$ be essential surfaces constructed in the proof of Theorem. 
Then each of them has a single boundary component. 
In particular, both are non-orientable. \qed
\end{corollary}

\subsection{Euler characteristics}

As explained in \cite[Section 3, subsection ``Euler characteristics'']{IMb}, 
we have a formula of the Euler characteristic $\chi$ for a type I candidate surface 
corresponding to an edgepath system $\Gamma = ( \gamma_1 , \cdots , \gamma_N)$ as follows. 

\begin{eqnarray}
\frac{-\chi}{\sharp s}
&=&
 \sum_{i=1}^{N}
  \left(
   \left\{
    \begin{array}{l}
     0 \\
     ~~~~~(\textrm{  if $\gamma_i$ is constant }) \\
     |\gamma_{i} |\\
     ~~~~~(\textrm{ otherwise }) \\
    \end{array}
   \right.
  \right)
\label{Eq:Formula:EulerCharTypeI}
\\
&& 
+N_{\mathrm{const}}-N
+\left(
N-2-\sum_{\gamma_i \in \Gamma_{\mathrm{const}} }\frac{1}{q_i}
\right)
\frac{1}{1-u}
,
\nonumber
\end{eqnarray}
where $\sharp s$ denotes the number of sheets, 
$N_{\mathrm{const}}$ the number of the constant edgepaths, 
$\Gamma_{\mathrm{const}}$ 
the set of constant edgepaths in $\Gamma$, 
and $u$ the common $u$-coodinate of the constant edgepaths. 

From this formula, we have the following by direct calculations. 

\begin{proposition}
For essential surfaces $F$ and $F'$ constructed in the proof of Theorem, 
the values $\frac{-\chi}{\sharp s}$ are both equal to 1. 
This implies that their Euler characteristics are $-n$ and $-n+7$ respectively, 
and are equal to the denominators of their boundary slopes. 
\end{proposition}

Some examples of the surfaces with the last property were obtained in \cite{Callahan13}. 
The above proposition says that 
$K_n = M( -1/2 , 2/5 , 1/n)$ with an odd positive integer $n \ge 11$ 
admits two distinct essential surfaces enjoying the property.

\section*{Acknowledgements}
The author would like to thank In Dae Jong, Hidetoshi Masai and Kimihiko Motegi for helpful discussions. 
He also thanks to the referees for careful readings.

\bibliographystyle{amsplain}

\end{document}